\documentclass[12pt]{amsart}

\usepackage{amsfonts}
\usepackage{amsmath}
\usepackage{amssymb}
\usepackage{latexsym}

\newtheorem{theorem}{Theorem}
\theoremstyle{plain}
\newtheorem{corollary}[theorem]{Corollary}
\newtheorem{definition}[theorem]{Definition}

\newtheorem{example}[theorem]{Example}

\newtheorem{remark}[theorem]{Remark}

\renewcommand\endproof{\hfill $\Box$\vskip 0.15in}

\def\pr{\noindent {\bf {Proof.\ }}}

\numberwithin{equation}{section}
\numberwithin{theorem}{section}

\newcommand\bel[1]{\begin{equation}\label{#1}}
\newcommand\ee{\end{equation}}

\newcommand\und[1]{{\tt{#1}}}

\newcommand{\ol}[1]{{\overline{#1}}}
\newcommand{\wt}[1]{{\widetilde{#1}}}

\newcommand{\bD}{{\mathbb{D}}}

\newcommand{\bR}{{\mathbb{R}}}
\newcommand{\bE}{{\mathbb{E}}}
\newcommand{\bP}{{\mathbb{P}}}
\newcommand{\bI}{{\mathbb{I}}}

\newcommand{\bA}{{\mathbf{V}}}

\newcommand{\bH}{{\mathbf{H}}}

\newcommand{\HX}{{\mathbb{H}}_{\mfX}}
\newcommand{\HXt}{\HX^{(1)}}
\newcommand{\HB}{{\mathbb{H}}_{\mfB}}
\newcommand{\HBt}{{\HB^{(1)}}}

\newcommand{\cF}{{\mathcal{F}}}
\newcommand{\cI}{{\mathcal{I}}}
\newcommand{\cJ}{{\mathcal{J}}}
\newcommand{\cK}{{\mathcal{K}}}
\newcommand{\cP}{{\mathcal{P}}}
\newcommand{\cR}{{\mathcal{R}}}

\newcommand{\mfX}{{\mathfrak{X}}}
\newcommand{\mfB}{{\mathfrak{B}}}

\def\Ito{It\^{o}}
\def\bola{\boldsymbol{\alpha}}
\def\bole{\boldsymbol{\epsilon}}
\def\boz{\boldsymbol{(0)}}
\def\nk{\mathfrak{K}}

\begin{document}

\today

\title[Stochastic Integrals and Evolution Equations]
{Stochastic Integrals and Evolution Equations with Gaussian Random
Fields}
\author{S. V. Lototsky}
\curraddr[S. V. Lototsky]{Department of Mathematics, USC\\
Los Angeles, CA 90089}
\email[S. V. Lototsky]{lototsky@math.usc.edu}
\urladdr{http://www-rcf.usc.edu/$\sim$lototsky}
\author{K. Stemmann}
\curraddr[K. Stemmann]{Department of Mathematics, USC\\
Los Angeles, CA 90089}
\email[K. Stemmann]{stemmann@usc.edu}
\thanks{Research was partially supported by the NSF Grant
DMS-0237724}

\begin{abstract}
The paper studies stochastic integration with respect to Gaussian processes and
fields. It is more convenient to work with a field than a process:
 by definition, a field is
  a collection of stochastic integrals for a class of deterministic integrands.
  The problem is then to extend the definition to random integrands.
 An orthogonal decomposition of the
chaos space of the random field, combined with the
Wick product,  leads to
 the \Ito-Skorokhod integral,
  and provides an efficient tool to
study the integral, both analytically and numerically.
For a Gaussian process, a natural definition of the integral follows
from  a canonical
correspondence between random processes and a special class of random
fields. Some examples of the corresponding stochastic differential
equations are also considered.
\end{abstract}

\maketitle

\section{Introduction}
While stochastic integral with respect to a standard Brownian
motion is a well-studied object,  integration
with respect to other Gaussian processes is currently an area of active research,
and the  fractional Brownian motion is receiving most of the attention
\cite[etc.]{AMN1,A,DH,DU,DHP,KBR,Lin,PT}
The objective of this paper is
to define and investigate stochastic
integrals with respect to arbitrary
Gaussian processes and fields using chaos expansion.

A generalized Gaussian field $\mfX$ over a Hilbert space $\bH$ is a continuous linear
mapping $f\mapsto \mfX(f)$ from $\bH$ to the space of Gaussian random variables.
 The corresponding
chaos space $\HX$ is the Hilbert space of
square integrable random variables that are measurable with respect to the
sigma-algebra generated by $\mfX(f),\ f\in \bH$. The  chaos expansion
is an orthogonal decomposition of $\HX$: given an orthonormal
basis $\{\xi_m,\, m\geq 1\}$ in $\HX$, a square integrable
$\bH$-valued random variable $\eta$ has the chaos expansion
$\eta=\sum_{m\geq 1} \eta_m\xi_m$, with $\eta_m=\bE(\eta\xi_m)\in \bH$.

The definition of a generalized Gaussian field $\mfX$ already provides the
stochastic integral $\mfX(f)$ for non-random $f\in \bH$. As a result, given the
chaos expansion of a random element $\eta$ from $\HX$,
the definition of the stochastic integral $\mfX(\eta)$
requires an extension of the linearity property of $\mfX$ to linear
combinations with random coefficients.
An extensions of this
property using the Wick product lead to the \Ito-Skorokhod integral
$\mfX^{\diamond}(f)$;
see Definition \ref{def:main} below.
Under some conditions, the integral coincides with the
divergence operator (the adjoint of the Malliavin derivative)
 on the chaos space $\HX$.

Even for non-random $f$ there are
often several ways of computing $\mfX(f)$.
It is most convenient to work with a white noise over
 $\bH$, that is, a zero-mean generalize Gaussian field
such that $\bE\big(\mfX(f)\mfX(g)\big)=(f,g)_{\bH}$ for all $f,g\in \bH$.
It turns out that, for  every zero-mean Gaussian field $\mfX$ over $\bH$,
 there exists a different (usually larger) Hilbert space $\bH'$ such that $\mfX$ is a
white noise over $\bH'$.
Moreover, the space $\bH'$ is uniquely determined by $\mfX$.
On the other hand, every
 zero-mean Gaussian field $\mfX$ over $\bH$ can be written in the form
 $\mfX(f)=\mfB(\cK^*f)$, $f\in \bH$,
 where $\cK^*$ is a bounded linear operator on $\bH$ and $\mfB$ is a white noise
 over $\bH$, although this white noise
 representation of $\mfX$ is not necessarily unique.
 Thus, different white noise representations of $\mfX$ lead to
 different formulas for computing $\mfX(f),$ and the
 chaos expansion is an efficient way for deriving those formulas.
 In particular, for both deterministic and random $f$, chaos expansion
 provides an explicit formula for $\mfX(f)$ in terms of the Fourier
 coefficients of the integrand $f$.

To define stochastic integral with respect to a
 Gaussian process $X=X(t)$, $t\in [0,T]$,
we construct a Hilbert space $\bH_X$ and a white noise $\mfB$ over
$\bH_X$ such that
 $X(t)=\mfB(\chi_t)$, where $\chi_t$ is the indicator
function of the interval $[0,t]$. The space $\bH_X$ is
 uniquely determined by $X$;
for example, the Wiener process on $(0,T)$ has $\bH_X=L_2((0,T))$.
 Then the equality
\bel{eq:def1}
\int_0^Tf(s)\diamond dX(s)=\mfB^{\diamond}(f),\ f\in \HB,
\ee
  defines the stochastic integral with respect to $X$.

  In some situations, given a Gaussian process $X=X(t),\ t\in [0,T]$,
  it is possible to find a
generalized Gaussian field $\mfX$ over a Hilbert space
$\bH$ so that $X(t)=\mfX(\chi_t)$. Even though  $\mfX$ is not
necessarily a white noise over $\bH$, the resulting definition
of the stochastic integral,
$$
\int_0^Tf(t)\diamond dX(t)=\mfX^{\diamond}(f),
$$
 coincides with the (\ref{eq:def1}), while
the space $\bH$ can be more
convenient for computations than the space $\bH_X$. For example,
fractional Brownian motion  with the Hurst parameter bigger then $1/2$
has a rather complicated space  $\bH_X$, but
can be represented using a generalized Gaussian field over
$\bH=L_2((0,T)).$

The paper is organized as follows.
Section \ref{sec:BG} provides background on generalized
Gaussian fields, the chaos expansion, and the Wick product.
Section \ref{sec:connect} establishes a connection
between Gaussian processes and fields.
 Section \ref{sec:SI} investigates the \Ito-Skorokhod
 stochastic integral. Section \ref{sec:see1} studies
the corresponding differential equations that admit a closed
form solution. Section \ref{sec:see2} studies more general
stochastic evolution equations and establishes the corresponding
stochastic parabolicity conditions.

\newpage

The main contributions of the paper are:
\begin{enumerate}
\item  A connection between  ge\-ne\-ra\-lized Gaussian
fields over $L_2((0,T))$ and processes that are representable in the
form $\int_0^t K(t,s)dW(s)$ (Theorem \ref{th:adapt});
\item Chaos expansions of the \Ito-Skorokhod
integral (Theorem \ref{th:st-in-ce});
\item Investigation of the equation $u(t)=1+\int_0^tu(s)\diamond
dX(s)$ for
a class of Gaussian random processes $X$ (Theorem \ref{th:sode1}).
\item A generalization of the stochastic parabolicity condition
(Theorem \ref{th:main-par}).
\end{enumerate}

 In particular, we establish the following result.

 \begin{theorem}
Let $\mfX$ be a zero-mean generalized Gaussian field over
$L_2((0,T))$ and $X(t)=\mfX(\chi_t)$. Then

(a) The solution of the \Ito-Skorokhod equation
$$
u(t)=1+\int_0^tu(s)\diamond dX(s)
$$ is unique in the class of square integrable $\cF^X$-measurable
processes and is given by
$$
u(t)=e^{X(t)-\frac{1}{2}\bE X^2(t)}.
$$

(b) The \Ito-Skorokhod partial differential equation
$$
du(t,x)=a\int_0^tu_{xx}(s,x)ds+\sigma\, u_x(t,x)\diamond dX(t)
$$
is well-posed in $L_2(\Omega; L_2(\mathbb{R}))$ if and only if
$$
at\geq\frac{\sigma^2}{2}\,\mathbb{E}X^2(t).
$$
\end{theorem}

\section{Generalized Gaussian Fields: A Background}
\label{sec:BG}

Let $(\Omega,\,\cF,\, \bP)$ be a probability space and $\bA$, a
linear topological space over the real numbers $\bR$.
Everywhere in this
paper, we assume that
the probability space is rich enough to support all the random
elements we might need.

\begin{definition}\label{Ch1:def1}
(a)  A \und{generalized random field} over $\bA$ is a
mapping \\
$\mfX: \Omega\times \bA \to \bR$ with the following properties:
\begin{enumerate}
\item For every $f\in \bA$, $\mfX(f)=\mfX(\cdot,f)$ is a random variable;
\item For every $\alpha, \beta\in \bR$ and $f,g\in \bA$,
$\mfX(\alpha f + \beta g)=\alpha \mfX(f)+\beta \mfX(g)$;
\item If \ $\lim\limits_{n\to \infty} f_n=f$ in the topology of $\bA$, then
$\lim\limits_{n\to \infty} \mfX(f_n)=\mfX(f)$ in probability.
\end{enumerate}

(b)  A generalized random field $\mfX$ is called
\begin{itemize}
\item \und{zero-mean}, if $\ \bE \mfX(f)=0$ for all $f\in \bA$;
\item  \und{Gaussian},
 if the random variable $\mfX(f)$ is Gaussian for
every $f\in \bA$.
\end{itemize}
\end{definition}

{\sc For Example}, if $W=W(t),\, 0\leq t\leq T,$ is a standard Brownian motion on
$(\Omega, \, \cF,\, \bP)$, then $\mfX(f)=\int_0^Tf(t)dW(t)$ is
a zero-mean generalized Gaussian field over $L_2((0,T))$; note that
\bel{ch1.1}
\bE|\mfX(f_n)-\mfX(f)|^2=\int_0^T|f_n(t)-f(t)|^2dt.
\ee
More generally, if ${\mathcal{M}}$ is  a bounded linear operator on $L_2((0,T))$,
then
\bel{ch1.2}
\mfX(f)=\int_0^T({\mathcal{M}}f)(t)dW(t)
\ee
is a zero-mean generalized Gaussian field over $L_2((0,T))$.
In fact, by
 Theorem \ref{th:wn11}(b) below,
every zero-mean generalized Gaussian field over $L_2((0,T))$
can be represented in the
form (\ref{ch1.2}) with suitable ${\mathcal{M}}$ and $W$.
We will also see
that the fractional Brownian motion on $[0,T]$
with Hurst parameter bigger than $1/2$
can be interpreted as a zero-mean generalized Gaussian field over $L_2((0,T)).$

 Let $\bH$ be a real Hilbert space with inner product $(\cdot, \cdot)_{\bH}$ and
 norm $\|\cdot\|_{\bH}=\sqrt{(\cdot,\cdot)_{\bH}}$. The following result
 is a direct consequence of the Riesz
Representation Theorem.

\begin{theorem}\label{th:gf}
For every zero-mean generalized Gaussian field $\mfX$  over a Hilbert space $\bH$,
 there exists a unique bounded linear self-adjoint operator $\cR$ on $\bH$ such that
\bel{ch1.3}
\bE\big( \mfX(f)\mfX(g)\big)=(\cR f,g)_{\bH},\ f,g \in \bH.
\ee
\end{theorem}

\begin{definition} (a) The operator $\cR$ from Theorem \ref{th:gf}
is called the \und{covariance operator} of $\mfX$.
(b) The field $\mfX$ is called {\tt non-degenerate} if
$\cR$ is one-to-one.
(c)  A \und{white noise} over $\bH$ is a zero-mean generalized Gaussian field with
the covariance operator equal to the identity operator.
\end{definition}

Standard arguments from functional analysis lead to the following
result.
\begin{theorem}\label{th:wn}
(a) For every zero-mean generalized Gaussian field  $\mfX$ over a Hilbert space $\bH$,
there exist a bounded linear  operator $\cK$ on $\bH$
 and a
white noise $\mfB$ over $\bH$ so that $\cK\cK^*$ is the covariance operator of $\mfX$
and, for every $f\in \bH$,
\bel{wn11}
\mfX(f)=\mfB(\cK^* f);
\ee
as usual, $\cK^*$ denotes the adjoint of $\cK$.

(b) For every zero-mean non-degenerate
generalized Gaussian field $\mfX$ over a Hilbert space $\bH$,
there exists a Hilbert space $\bH_{\cR}$ such that
$\bH$ is continuously embedded into $\bH_{\cR}$ and
$\mfX$ extends to a white noise over $\bH_{\cR}$.
\end{theorem}

\begin{remark}\label{rem-wn}
 (a) If $\mfX$ is non-degenerate and $\cR: \bH\to\bH$
 is onto, then $\cR$ has a bounded inverse and
  $\bH_{\cR}=\bH$.
  (b) If $\ker{\cR}$ is non-trivial, then
we can define $\bH_{\cR}$ as the closure of the factor space $\bH/\ker(\cR)$
 with respect to the inner product $(\ol{f},\ol{g})_{\bH_{\cR}}=(\cR f,g)_{\bH}$,
 where $\ol{f}$ is the equivalence class of $f$ in $\bH/\ker(\cR)$.
 Direct computations show that
 the  generalized random field $\mfB$ over $\bH/\ker(\cR)$, defined by
 $$
 \mfB(\ol{f})=\mfX(f), \ f\in \bH,
 $$
 extends to a white noise over $\bH_{\cR}$.
 \end{remark}

We will now discuss several connections between generalized
Gaussian fields  and Gaussian processes.

 Denote by $\chi_t=\chi_t(s)$  the
indicator function of the interval $[0,t]$:
\bel{char-f}
\chi_t(s)=
\begin{cases}
1, & 0\leq s \leq t;\\
0, & {\rm otherwise}.
\end{cases}
\ee
With this definition, $\chi_{t_2}(s)-\chi_{t_1}(s)$ is the
indicator  function of the interval $(t_1,t_2]$, $t_2>t_1$.

\begin{theorem} \label{th:wn11}
 (a) If $\mfB$ is a white noise over $L_2(\bR)$,
 then $B(t)=\mfB(\chi_t)$ is a standard Brownian motion on
$(\Omega, \, \cF,\, \bP)$ and, for every $f\in L_2(\bR)$, we have
\bel{wn-rep1}
\mfB(f)=\int_{\bR}f(s)dB(s).
\ee
(b) For every zero-mean non-degenerate generalized Gaussian field
$\mfX$ over $L_2((0,T))$, there exist a bounded linear operator $\cK^*$
on $L_2(\bR)$ and a
standard Brownian motion $W=W(t)$ such that, for every $f\in L_2(\bR)$,
\bel{wn-rep2}
\mfX(f)=\int_{\bR}(\cK^*f)(s)dW(s).
\ee
\end{theorem}

\pr (a) Direct computations.

 (b) This follows from part (a) and from Theorem \ref{th:wn}.
 \endproof

 \begin{remark} While beyond the scope of this paper,
 a similar result is true for multi-parameter
processes as well. For example, if $\mfX$ is a generalized Gaussian
field
over $L_2(\mathbb{R}^2)$, then the same arguments show that
\eqref{wn-rep2} holds with a Brownian sheet $W$.
\end{remark}

 {\tt  From now
on, we assume that the space $\bH$ is separable, the
field $\mfX$ is \\ non-degenerate, and
$\cF=\cF^{\mfX}$, the sigma-algebra generated by the random variables
$\mfX(f),\ f\in \bH$.}

\begin{definition} (a) The \und{chaos space} generated by $\mfX$ is the
collection of all square-integrable
random variables on $(\Omega,\, \cF,\, \bP)$.  This chaos space
will be denoted by $\HX$.

(b) The \und{first chaos space} generated by $\mfX$ is the sub-space of $\HX$,
consisting of the random variables $\mfX(f)$, $f\in \bH$.
 The first chaos space will be denoted by $\HXt$.
\end{definition}

It follows that $\HX$ is a Hilbert space with inner product $(\xi,\eta)_{\HX}=
\bE(\xi\eta)$, and $\HXt$ is a Hilbert sub-space of $\HX$.
Moreover, the space $\HXt$ is separable: if
$\{\bar{f}_1, \bar{f}_2, \ldots\}$ is
 a dense countable set in $\bH$, then the collection of all finite linear combinations
of $\mfX(\bar{f}_i)$ with rational coefficients is a dense countable set in
$\HXt$.

Our next objective is to show how an orthonormal basis in $\HXt$ leads to an
orthonormal basis in $\HX$. We will need some additional constructions.

For an integer $n\geq 0$,
  the $n$-th \und{Hermite polynomial} $H_n=H_n(t)$ is defined by
\begin{equation}
\label{eq:HerPol}
 H_{n}(t)=(-1)^ne^{t^{2}/2}
\frac{d^{n}}{dt^{n}}e^{-t^{2}/2}.
\end{equation}
Next, denote by $\cI$ the collection of
multi-indices, that is, sequences
 $\alpha=\{\alpha_k,\ k\geq 1\}= \{\alpha_1, \alpha_2,\ldots\}$
 with the following properties:
 \begin{itemize}
 \item each $\alpha_k$ is a non-negative integer: $\alpha_k\in \{0,1,2,\ldots\}$.
 \item only finitely many of $\alpha_k$ are non-zero:
 $|\alpha|:=\sum\limits_{k=1}^{\infty} \alpha_k < \infty. $
 \end{itemize}
 The set $\cI$ is countable, being a countable union of countable sets.
  By $\epsilon_n$ we denote the multi-index $\alpha=\{\alpha_k,\ k\geq 1\}$
 with $\alpha_k=1$ if $n=k$ and $\alpha_k=0$ otherwise.
 For $\alpha \in \cJ$, we will use the notation
 $\alpha!:=\alpha_1!\,\alpha_2!\cdots.$
 Let $\{\xi_1,\xi_2,\ldots\}$ be an ordered countable collection of random variables.
 For $\alpha\in \cI$ define random variables $\xi_{\alpha}$ as follows:
 \bel{xi-al}
 \xi_{\alpha}=\prod_{k\geq 1}\frac{H_{\alpha_k}(\xi_k)}{\sqrt{\alpha_k!}},
 \ee
 where $H_{\alpha_k}$ is $\alpha_k$-th Hermite polynomial (\ref{eq:HerPol}).

 \begin{theorem}
 \label{th:basis}
 Let $\{\xi_1,\xi_2,\ldots\}$ be an orthonormal basis in $\HXt$.
 Then the collection $\Xi=\{\xi_{\alpha}, \ \alpha \in \cI\}$
 is an orthonormal basis in $\HX$: for every $\eta\in \HX$ we have
 $$
 \eta=\sum_{\alpha \in\cI} \Big(\bE(\eta\xi_{\alpha})\Big)\, \xi_{\alpha},\ \
 \bE\eta^2=\sum_{\alpha\in \cI} \Big(\bE(\eta\xi_{\alpha})\Big)^2.
 $$
 \end{theorem}
\pr
See \cite[Theorem 2.1]{Maj}.
\endproof

\begin{corollary}
\label{prop:wne}
Let $\mfB$ be a white noise over a separable Hilbert space $\bH$ and let
$\{m_1,\, m_2,\, \ldots\}$ be an orthonormal basis in $\bH$. Then
$\{\xi_k=\mfB(m_k),\ k\ge 1\}$ is an orthonormal basis in $\HBt$ and, for
every $f\in \bH$,
\bel{wne}
\mfB(f)=\sum_{k=1}^{\infty} (f,m_k)_{\bH}\, \mfB(m_k).
\ee
\end{corollary}

\pr Note that $\bE \big(\xi_k \xi_n\big)=\bE \big(\mfB(m_k)\mfB(m_n)\big)
=(m_k,m_n)_{\bH}$, so
the system $\{\xi_k,\ k\ge 1\}$ is orthonormal in $\HBt$
 if and only if $\{m_k,\, k\geq 1\}$ is orthonormal in $\bH$.
If $\xi\in \HBt$, then $\xi=\mfB(f)$ for some $f\in \bH$.
By assumption, $f=\sum_{k=1}^{\infty}(f,m_k)_{\bH}\, m_k$, which implies
(\ref{wne}) and completes the proof. \endproof

If $\bH=L_2((0,T))$ and $f=\chi_t$, then (\ref{wne}) becomes
a familiar representation of the standard Brownian motion on $[0,T]$:
\bel{wn:wp}
W(t)=\sum_{k=1}^{\infty} \left(\int_0^t m_k(s)ds\right) \,
\left(\int_0^T m_k(s)dW(s)
\right).
\ee
Now, let $\mfX$ be a zero-mean generalized Gaussian field over
a separable Hilbert space $\bH$.
By (\ref{wne}) and Theorem \ref{th:wn}(a),
we can take a white noise representation of $\mfX$,
$\mfX(f)=\mfB(\cK^* f)$, and get an expansion of $\mfX(f)$ using an
orthonormal basis in $\bH$:
\bel{wne1}
\mfX(f)=\sum_{k=1}^{\infty} (\cK^* f , m_k)_{\bH}\, \mfB(m_k).
\ee
Alternatively, by Theorem \ref{th:wn}(b)
$\mfX$ is a white noise over the space $\bH_{\cR}$ corresponding to the
covariance operator $\cR$ of $\mfX$.
 If $\{\ol{m}_k, \, k\geq 1\}$
is an orthonormal basis in $\bH_{\cR}$, then
we have an equivalent expansion of $\mfX(f)$:
\bel{nwne}
\mfX(f)=\sum_{k=1}^{\infty} (\cR f,\ol{m}_k)_{\bH}\, \mfX(\ol{m}_k).
\ee

We conclude the section with a brief discussion of the Wick
product, as we will
need this product to define $\mfX(f)$ for random $f$.
For more details, see \cite{HOUZ}.

The Wick product of two
arbitrary elements of $\HX$ can be computed by the formula
\bel{WP-def}
\left(\sum_{\alpha\in \cI} c_{\alpha}\xi_{\alpha}\right) \diamond
\left(\sum_{\beta\in \cI} d_{\beta}\xi_{\beta}\right)=
\sum_{\alpha,\beta\in \cI}c_{\alpha}d_{\beta}
\sqrt{\frac{(\alpha+\beta)!}{\alpha!\beta!}}\, \xi_{\alpha+\beta},
\ee
where $\alpha+\beta=\{\alpha_k+\beta_k,\, k\geq 1\}$ and
$\alpha!=\prod_{k\geq 1}\alpha_k!=\alpha_1!\alpha_2!\alpha_3!\cdots$.
In general, there is no guarantee that, for  $\xi,\eta\in \HX$, the
Wick product $\xi\diamond \eta $ belongs to $\HX$.

Similar to ordinary powers, we define Wick powers of a random variable
$\eta\in \HX$: $\eta^{\diamond n}=\eta\diamond\cdots\diamond \eta$.
Replacing ordinary powers with Wick powers in a Taylor series
for a function
$f$  leads
to the notion of a Wick function $f^{\diamond}$. For example, the Wick
exponential $e^{\diamond \eta}$ is defined by
\bel{Wick-exp}
e^{\diamond \eta}=\sum_{n=1}^{\infty} \frac{\eta^{\diamond n}}{n!}
\ee
and satisfies $e^{\diamond (\xi+\eta)}=e^{\diamond \xi}\diamond e^{\diamond \eta}.$
If $\eta\in\HXt$, then direct computations show that
\bel{WickExp}
e^{\diamond \eta}=e^{\eta-\frac{1}{2}\bE\eta^2}.
\ee

\section{Connection Between Processes and Fields}
\label{sec:connect}
Given a zero-mean
 generalized Gaussian field $\mfX$ over $L_2((0,T))$,
 we define its
 \und{associated process} $X(t),\ t>0,$ by
 \bel{a-proc}
 X(t)=\mfX(\chi_t).
 \ee
 Clearly, $X(t)$ is a Gaussian process. Let $\cK^*$ be the operator
 from Theorem \ref{th:wn11} and define the kernel function
 $K_{\mfX}=K_{\mfX}(t,s)$ by
 \bel{functionK}
K_{\mfX}(t,s)=(\cK^*\chi_t)(s).
\ee
 It then follows from (\ref{wn-rep2}) that
 \bel{a-proc1}
 X(t)=\int_0^T K_{\mfX}(t,s)dW(s)
 \ee
 for some standard Brownian motion $W$. Let us emphasize that, while
 every kernel $K(t,s)$ with minimal integrability properties can
 define a Gaussian process according to (\ref{a-proc1}), only a process
 associated with a generalized field over $L_2((0,T))$ has
 a kernel defined according to (\ref{functionK}), where $\cK^*$ is a
 bounded operator on $L_2((0,T))$. Recall that the definition
 of a generalized field (Definition \ref{Ch1:def1})
 includes a certain continuity property, and this property translates into addition
 structure of the kernel function in the representation of the associated
process.

Now assume that we are given a Gaussian process $X(t)$ defined by
(\ref{a-proc1}) with some kernel $K_{\mfX}(t,s)$.
 {\em We are not assuming
that $K_{\mfX}$ has the form \eqref{functionK}}.  In what follows, we
discuss sufficient conditions on $K_{\mfX}(t,s)$ ensuring that
$X(t)$ is the associated process of a generalized Gaussian
 field $\mfX$ over $L_2((0,T))$, that is,
representation (\ref{functionK}) does indeed hold with some
 bounded linear operator $\cK^*$ on $L_2((0,T))$.
For that, we need to  recover the operator $\cK^*$
from the kernel $K_{\mfX}(t,s)$.  By linearity, if  (\ref{functionK}) holds and
 if $s_0<s_1<\ldots<s_N$ are points in $[0,T]$ and
\bel{stepF}
f(s)=\sum\limits_{k=0}^{N-1} a_k(\chi_{s_{k+1}}(s)-\chi_{s_{k}}(s))
\ee is a step function, then
\bel{op1}
\cK^* f(s)= \sum_{k=0}^{N-1} a_k \big(K_{\mfX}(s_{k+1},s)-K_{\mfX}(s_k,s)\big).
\ee
To extend (\ref{op1}) to continuous functions $f$, the kernel
$K_{\mfX}(t,s)$ must have bounded variation as a function of $t$;
if this is indeed the case, then (\ref{op1}) implies that, for every
smooth compactly supported function $f$ on $[0,T]$,
\bel{op2}
\cK^* f(s)=\int_0^T f(t)K_{\mfX}(dt,s).
\ee
 It now follows that if
 the partial derivative $\partial K_{\mfX}(t,s)/\partial t$ exists and
 is square integrable over $[0,T]\times [0,T]$, then $\cK^*$,
 as defined by (\ref{op2}), extends to a bounded linear operator on
 $L_2((0,T))$.

 Let us now assume that  the process $X(t)$ define by (\ref{a-proc1})
 is \und{non-anticipating},  i.e.
 adapted to the filtration $\{\cF_t^W,\ 0\leq t\leq T\}$
  generated by the Brownian motion $W(s).$
  Then $K_{\mfX}(t,s)=0$ for $s>t$ and
 (\ref{a-proc1}) becomes
 \bel{a-proc2}
 X(t)=\int_0^t K_{\mfX}(t,s)dW(s).
 \ee
 Note that in this case we have
 \bel{cov-funct}
 \bE\big(X(t)X(s)\big)=\int_0^{\min(t,s)}K_{\mfX}(t,\tau)K_{\mfX}(s,\tau) d\tau.
 \ee
 For such processes,
 formula (\ref{op1}) and the conditions for the continuity of the
 corresponding operator $\cK^*$ must be modified as follows.

 \begin{theorem}\label{th:adapt}
 Assume that
 the process $X(t)$ defined by (\ref{a-proc1}) is non-anticipating.

 (a) If $f$ is a step function (\ref{stepF}), then
 \bel{KK}
\begin{split}
\cK^*f(s)&=\sum_{i=0}^{N-1}\big(\chi_{s_{i+1}}(s)-\chi_{s_i}(s)\big)
\Big(a_iK_{\mfX}(s_{i+1},s)\\
&+\sum_{k=i+1}^{N-1}a_{k}\big(K_{\mfX}(s_{k+1},s)-
K_{\mfX}(s_{k},s)\big)
\Big).
\end{split}
\ee
(b) If the function $K_{\mfX}(\cdot, s)$ has bounded
variation for every $s$ and  $\lim\limits_{\delta\to 0,\,
\delta>0}K_{\mfX}(s+\delta,s)=K_{\mfX}(s^+,s)$
exists for all $s\in (0,T),$
 then
  \bel{ker-op000}
   \cK^*f(s)=K_{\mfX}(s^+,s)f(s)+\int_s^T f(t)K_{\mfX}(dt,s)
 \ee
 for every continuous on $[0,T]$ function $f$.

 (c) If the function
$K_{\mfX}(t,s)$ has the following properties
\begin{enumerate}
\item $K_{\mfX}$ is continuous and non-negative for
 $0\leq s\leq t \leq T$,
 and $\sup\limits_{0<t<T} K_{\mfX}(t,t) \leq K_0$;
\item For every fixed $s_0$, the function $K_{\mfX}(t,s_0)$
is monotone as a function of $t$ and the partial
derivative $K^{(1)}(t,s)=\partial K_{\mfX}(t,s)/\partial  t$
exists for all $0<s<t<T$;
\item There exists a number $K_1=K_1(T)$ such that \bel{ker-cond1}
\sup_{0<t<T} \int_0^t K_{\mfX}(T,s)  |K^{(1)}(t,s)| ds
\leq K_1^2,
\ee
\end{enumerate}
then the corresponding operator $\cK^*$ defined by equation
\eqref{ker-op000} is bounded on $L_2((0,T))$ and the operator norm
$\|\cK^*\|$ of $\cK^*$ satisfies
\bel{op-norm}
\|\cK^*\|^2 \leq (K_0+K_1)^2.
\ee
\end{theorem}

\pr (a) By assumption, $K_{\mfX}(t,s)=0$ for $s>t$.
 Fix an $s$ such that $s\in (s_j,s_{j+1}]$ for some $j=0,\ldots, N-1$.
  By (\ref{op1}) we have for this value of $s$
\begin{equation*}
\begin{split}
\cK^*f(s)&=\sum_{k=0}^{N-1} a_k \big(K_{\mfX}(s_{k+1},s)-
K_{\mfX}(s_{k},s) \big)\\
&=a_{j}K_{\mfX}(s_{j+1},s)
+\sum_{k=j+1}^{N-1}a_{k}\big(K_{\mfX}(s_{k+1},s)-
K_{\mfX}(s_{k},s)\big).
\end{split}
\end{equation*}
Since $\chi_{s_{k+1}}(s)-\chi_{s_{k}}(s)$ is the indicator function
of the interval $(s_k,s_{k+1}]$, (\ref{KK}) follows.

(b) Under the additional assumptions on the kernel $K_{\mfX}$,
 (\ref{ker-op000}) follows from (\ref{KK}) after passing to the limit
 $\max\limits_{j=0,\ldots, N-1}|s_{j+1}-s_j|\to 0$.

 (c) With no loss of generality, we can assume that
 $K_{\mfX}(s^{+},s)=0$ and
 $K^{(1)}\geq 0$; otherwise, we replace $K_{\mfX}(t,s)$ with
 either $K_{\mfX}(s^{+},s)-K_{\mfX}(t,s)$
 or $K_{\mfX}(t,s)-K_{\mfX}(s^{+},s)$.
  Let $g$ be a smooth compactly supported function on $(0,T)$.
 It follows from (\ref{ker-op000}) that
 \begin{equation*}
 \cK^*g(s)=
\int_s^T K^{(1)}(\tau,s)\, g(\tau)\, d\tau.
 \end{equation*}
  Then
   we use the Cauchy-Schwartz
  inequality and the properties of $K^{(1)}$:
  \begin{equation*}
\begin{split}
&\int_0^T \left| \int_s^T  K^{(1)}(\tau,s)
                g(\tau) d\tau \right|^2 ~ds
                = \int_0^T
                \left| \int_s^T \left[K^{(1)}(\tau,s)\right]^{1/2}
                \left[K^{(1)}(\tau,s)\right]^{1/2} g(\tau) d\tau
                 \right|^2 ~ds \\
            &\leq \int_0^T \int_s^T K^{(1)}(\tau,s) d\tau
                \int_s^T K^{(1)}(\tau,s) g^2(\tau) d\tau ~ds\\
           & \leq \int_0^T
           \left( K_{\mfX}(T,s) - K_{\mfX}(s,s)\right) \int_s^T
                K^{(1)}(\tau,s) g^2(\tau) d\tau ~ds \notag\\
            &\leq \int_0^T \left( \int_0^\tau K_{\mfX}(T,s)
                K^{(1)}(\tau,s) ~ds \right) g^2(\tau) ~d\tau \notag
    \leq K_1^2(T) \|g\|_{L_2((0,T))}^2.
    \end{split}
    \end{equation*}
    \endproof
Here are several examples of processes covered by
  part (c) of  Theorem \ref{th:adapt}.
\begin{example} {\rm
Assume that $K_{\mfX}(s^{+},s)=0$ and $K^{(1)}(t,s)\geq 0$.
Let $R(t,s)=\mathbb{E}\big(X(t)X(s)\big)$. By \eqref{cov-funct},
$$
\frac{\partial R(T,t)}{\partial t}=
\int_0^t K_{\mfX}(T,s)  K^{(1)}(t,s) ds
$$
and therefore
\begin{equation}
\label{fBM-00c}
K_1^2(T)=\sup_{0<t<T}\frac{\partial R(T,t)}{\partial t}.
\end{equation}
In particular, for
the fractional Brownian motion $W^H$ on $[0,T]$ with the
 Hurst parameter $H>1/2$,
 $$
 R(T,t)=\frac{1}{2}\left(T^{2H}+t^{2H}-(T-t)^{2H}\right),
 $$
and  (see Nualart \cite[Section 5.1.3]{Nualart})
 $W^H$ has representation (\ref{a-proc2}) with
$$
K_{\mfX}(t,s)=C_H\left(H-\frac{1}{2}\right)
 s^{\frac{1}{2}-H}\int_s^t (\tau-s)^{H-\frac{3}{2}}
 \tau^{H-\frac{1}{2}}\, d\tau,
$$
where
$$
C_H=\left(\frac{2H\Gamma\left(\frac{3}{2}-H\right)}{\Gamma\left(H+\frac{1}{2}\right)
\Gamma(2-2H)}\right)^{\frac{1}{2}}
$$
and $\Gamma$ is the Gamma-function. In this case, $K_{\mfX}(s^{+},s)=0$
 and $K^{(1)}(t,s)\geq 0$. By \eqref{fBM-00c},
\bel{fbm-bound}
K_1^2(T)=2H\,2^{1-2H} T^{2H-1}.
\ee
The bound  $K_1(T)$ is asymptotically optimal: if $H=1/2$,
which corresponds to the standard Brownian motion,
the right-hand side of (\ref{fbm-bound}) is equal to $1$.
}
\end{example}

\begin{example} {\rm
Assume that
$$
K_{\mfX}(t,s)=\rho\big((t-s)^{\alpha})\chi_t(s),\ \alpha>0,
$$
where $\rho$ is a non-negative, monotone,  continuously
differentiable function on $[0,T]$.

If $\rho$ is non-increasing,
then  $K_0=\rho(0)$ and
$K_1^2(T)=\rho(0)\big(\rho(0)-\rho(T^{\alpha})\big)$.
In particular, consider the {\tt stable  Ornstein-Uhlenbeck process}
\begin{equation}
\label{OU}
dX(t)=-bX(t)dt+dW(t), \ X(0)=0,\ b>0,
\end{equation}
so that $X(t)=\int_0^te^{-b(t-s)}dW(s)$ and
 $K_{\mfX}(t,s)=e^{-b(t-s)}$.  For this process,
 $\|\mathcal{K}^*\|\leq 1+\sqrt{1-e^{-bT}}$.
This bound is
 asymptotically optimal: as $b\to 0$, the process becomes $W$,
  and the upper bound on $\|\mathcal{K}^*\|$ becomes $1$.

If $\rho$ is non-decreasing, then  $K_0=\rho(0)$ and
$K_1^2(T)=\rho(T^{\alpha})\big(\rho(T^{\alpha})-\rho(0)\big)$.
In particular, consider the {\tt unstable  Ornstein-Uhlenbeck process}
\begin{equation}
\label{OU-u}
dX(t)=bX(t)dt+dW(t), \ X(0)=0,\ b>0,
\end{equation}
so that  $X(t)=\int_0^te^{b(t-s)}dW(s)$ and
 $K_{\mfX}(t,s)=e^{b(t-s)}$.  For this process,
 $\|\mathcal{K}^*\|\leq 1+\sqrt{e^{bT}(e^{bT}-1)}$.
This bound is
 asymptotically optimal: as $b\to 0$, the process becomes $W$,
  and the upper bound on $\|\mathcal{K}^*\|$ becomes $1$.
  }
\end{example}

 The following theorem establishes a connection between
 a zero-mean Gaussian process and white noise.

 \begin{theorem} \label{th:gp-wn} For every zero-mean  Gaussian process
 $X=X(t), \, t\in [0,T] $ with $X(0)=0$ and the covariance
 function $R(t,s)=\bE\big(X(t)X(s)\big)$, there exist
 \begin{enumerate}
 \item
 a Hilbert space $\bH_{R}$ containing the indicator functions $\chi_t;$
 \item  a white noise $\mfB$ over $\bH_{R}$
 \end{enumerate}
 such that $X(t)=\mfB(\chi_t)$.
 \end{theorem}

 \pr Let the Hilbert space
 $\bH_{R}$  be the closure of the set of the step functions with respect to
 the inner product
 $$
 (\chi_{t_1},\chi_{t_2})_{\bH_{R}}=R(t_1,t_2).
 $$
 Define a generalized
 Gaussian field $\mfB$ over $\bH_{R}$ by setting
 \bel{gen-f-p}
 \mfB(\chi_t)=X(t),
 \ee
   and then extending by linearity and
  continuity to all of $\bH_{R}$. With this definition, $\mfB$ is
 a white noise over $\bH_R$.
\endproof

By analogy with (\ref{wn-rep1}), if $\mfX$ is a generalized Gaussian field
over a Hilbert space $\bH$ containing $\chi_t$, $t\in [0,T]$ and
 $X(t)$ is the associated process
of $\mfX$, then $\int_0^T f(s)dX(s)$ can be an alternative notation
 for $\mfX(f)$.

 If $X(t)$ is the associated process of a zero-mean non-degenerate
 generalized Gaussian field
 $\mfX$ over $\bH=L_2((0,T))$, and $\cR$ is the covariance operator
  of $\mfX$, then
 $R(t,s)=(\cR \chi_t,\chi_s)_{L_2((0,T))}$ and the space $\bH_{R}$
 coincides with $\bH_{\cR}$ from Theorem \ref{th:wn}.

 If  $R(t,s)=\min(t,s)$, then
 $(\chi_{t_1},\chi_{t_2})_{\bH_{R}}
 =(\chi_{t_1},\chi_{t_2})_{L_2((0,T))}$.
 That is, for the Wiener process,  $\bH_R=L_2((0,T))$.
 For the fractional Brownian motion, the space $\bH_{R}$
 can be characterized using fractional derivative operators
 \cite{Nualart}.
 For a general  Gaussian process $X$ with
 covariance function $R$, an explicit
 characterization of the space $\bH_R$ is impossible.

\section{Stochastic Integration}
\label{sec:SI}

In the definition of a generalized random field $\mfX$ over a Hilbert space $\bH$,
 we consider random variables
$\mfX(f)$ for non-random $f\in \bH$. In this section, we define
$\mfX(\eta)$ for $\bH$-valued random elements $\eta$.

One possible way to proceed is to
write  the chaos expansion of $\eta$,
 $\eta=\sum_{\alpha}\eta_{\alpha}\xi_{\alpha}$, $\eta_{\alpha}\in \bH$,
 and then define $\mfX(\eta)$ as a corresponding linear combination of
 $\mfX(\eta_{\alpha})$.
 In the case $\bH=L_2((0,T))$, another possibility is to
 take a partition $0=t_0<t_1<\ldots<t_N=T$ of the interval
 $[0,T]$ and approximate $\eta(t)$ with a sum
 $\sum_{i=1}^N \eta(t_i^*)(\chi_{t_{i}}-\chi_{t_{i-1}})$,
 where $t_i^*\in [t_{i-1},t_i]$, and then approximate
 $\mfX(\eta)$ with the corresponding linear combination of
 $\mfX(\chi_{t_{i}})-\mfX(\chi_{t_{i-1}})$.

Either way, we need to address the following question.
By definition, if $\alpha,\beta$ are real numbers and $f,g$
are elements of $\bH$, then $\mfX(\alpha f+\beta g)=
\alpha\mfX(f)+\beta\mfX(g)$. But what if $\alpha$ and
$\beta$ are random variables? One possibility would be to
keep the same linearity. In the case
$\bH=L_2((0,T))$ this would
imply
\begin{equation}
\label{se-00005}
\mfX\left(\sum_{i=1}^N \eta(t_i^*)(\chi_{t_{i}}-\chi_{t_{i-1}})
\right)=
\sum_{i=1}^N\eta(t_i^*)(X(t_i)-X(t_{i-1})),
\end{equation}
where $X(t)=\mfX(\chi_t)$ is the associated process of $\mfX$.
 While natural, this extension of the
linearity property can lead to ambiguities in the definition
 of the corresponding stochastic integral. Indeed,
let $\mfB$ be a white noise over $L_2((0,T))$, and
let $\eta(t)=\mfB(\chi_t)$. By Theorem \ref{th:wn11},
$\eta$ is a standard Brownian motion $W$ and,
  as we know, the limit of the sum $\sum_{i=1}^N
  W(t_i^*)(W(t_{i+1}-W(t_i))$ depends on the
location of the points $t_i^*$.

Let us now consider an alternative to \eqref{se-00005}:
\begin{equation}
\label{se-00006}
\mfX\left(\sum_{i=1}^N \eta(t_i^*)(\chi_{t_{i}}-\chi_{t_{i-1}})
\right)=
\sum_{i=1}^N\eta(t_i^*)\diamond(X(t_i)-X(t_{i-1})).
\end{equation}
This time, if we take $\eta(t)=X(t)=W(t)$, a standard Brownian
motion, then the limit is the \Ito~integral $\int_0^TW(t)dW(t)$; it
is equal to $(W^2(T)-T)/2=(W(T))^{\diamond 2}/2$
 and does not depend on the location of the points $t_i^*$.

 Accordingly, we adopt the following convention: if
 $\mfX$ is a generalized Gaussian field over $\bH$,
 then, for every $f,g\in \bH$ and all random variables
 $\xi, \eta$,
 $$
 \mfX(\xi f+\eta g)=\xi\diamond \mfX(f)+\eta\diamond\mfX(g);
 $$
 recall that, by assumption, the underlying probability space
 $(\Omega, \mathcal{F},\mathbb{P})$ is such that
 $\mathcal{F}$ is generated by $\mfX$.
 With this convention, we proceed with the definition of the
 stochastic integral $\mfX(\eta)$ for $\bH$-valued random elements
 $\eta$ using the chaos expansion. There are at least three
   advantages of the chaos
 approach over time partitioning:
 \begin{enumerate}
 \item  generality: spaces
 other than $L_2((0,T))$ can be considered;
 \item possibility
 to use weighted chaos spaces, which  eliminates many questions about
 convergence;
 \item computational efficiency: the Wick product must only be
 computed for the basis elements $\xi_{\alpha}$.
 \end{enumerate}

 Because we are not
 restricted with the choice of $\bH$,
 we will assume that $\mfX=\mfB$, a white noise over $\bH$.
 Let $\{m_k,\ k\geq 1\}$ be an
orthonormal basis in $\bH$. Define $\xi_k=\mfB(m_k)$ and
 $\xi_{\alpha},\, \alpha \in \cI,$ according to (\ref{xi-al}).

 By Theorem \ref{th:basis},
 every $\bH$-valued random element $\eta$ with
 $\bE\|\eta\|_{\bH}^2<\infty$ has chaos expansion
\bel{rep-eta}
\eta=\sum_{\alpha\in \cJ}\eta_{\alpha} \xi_{\alpha}, \
\eta_{\alpha}=\bE(\eta\xi_{\alpha})\in \bH.
\ee

\begin{definition}\label{def:main}
Let $\eta$ be  an $\bH$-valued random element with chaos
expansion (\ref{rep-eta}).\\
The \und{\Ito-Skorokhod  stochastic integral} of $\eta$
with respect to $\mfB$ is
\bel{def:itoi}
\mfB^{\diamond}(\eta)=\sum_{\alpha\in \cI} \mfB(\eta_{\alpha})\diamond \xi_{\alpha},
\ee
where $\diamond$ is the Wick product.
\end{definition}

Since every generalized Gaussian field and every Gaussian process can be
represented using a white noise over a suitable Hilbert space,  formula
(\ref{def:itoi})  defines stochastic integral
with respect to any Gaussian process or field. We will see below that
this formula also
provides a chaos expansion of the integral in terms of the chaos expansion
of the integrand; note that  (\ref{def:itoi})
is not a chaos expansion in the sense of (\ref{rep-eta}).
 The two immediate question that are raised by the above definition
 and will be discussed below are
(a) the convergence of the series, and (b)
the dependence of the integral on the choice of the basis in $\bH$.

We start by deriving the chaos expansion of the integrals without
investigating the question of convergence.

\begin{theorem}\label{th:st-in-ce}
Let $\eta$ be  an $\bH$-valued random element
 with chaos expansion (\ref{rep-eta}),
and assume that
\bel{rep-eta-al}
\eta_{\alpha}=\sum_{k=1}^{\infty} \eta_{\alpha,k}m_k.
\ee
Then
\bel{ito-chaos}
\mfB^{\diamond}(\eta)=\sum_{\alpha\in \cI}\left( \sum_{k=1}^{\infty}
\sqrt{\alpha_k}\eta_{\alpha-\epsilon_k,k}\right) \xi_{\alpha}.
\ee
\end{theorem}

\pr By (\ref{rep-eta-al}) and linearity, keeping in mind that
both $\eta_{\alpha}$ and $m_k$ are non-random,
$$
\mfB(\eta_{\alpha})
=\sum_{k=1}^{\infty}\eta_{\alpha,k}\mfB(m_k)=
\sum_{k=1}^{\infty} \eta_{\alpha,k}\xi_k.
$$
Therefore,
\bel{ito-aux}
\mfB^{\diamond}(\eta)=\sum_{\alpha\in \cI}
\sum_{k=1}^{\infty} \eta_{\alpha,k}\xi_k\diamond \xi_{\alpha}=
\sum_{\alpha\in \cI}\sum_{k=1}^{\infty}
\sqrt{\alpha_k+1}\,\eta_{\alpha,k} \, \xi_{\alpha+\epsilon_k};
\ee
 recall
that $\epsilon_k$  is the multi-index with the only non-zero entry, equal to one,
at position $k$. By shifting the summation index,
 we get (\ref{ito-chaos}).
Note that, for every $\alpha \in \cI$, the inner sum in (\ref{ito-chaos})
contains finitely many non-zero terms.

\endproof

Now, let us address the questions of convergence and independence of basis.
The Cauchy-Schwartz inequality implies that if
\bel{D12}
\sum_{\alpha\in \cI} |\alpha|\, \|\eta_{\alpha}\|^2_{\bH}<\infty,
\ee
 then
$\mfB^{\diamond}(\eta)\in \HB$. Further examination of (\ref{ito-chaos})
shows that, for every  $\eta$
 satisfying (\ref{D12}),
$\mfB^{\diamond}(\eta)=\boldsymbol{\delta}(\eta)$, where
 $\boldsymbol{\delta}$ is the divergence operator
 (adjoint of the Malliavin derivative),
 and therefore $\mfB^{\diamond}(\eta)$
does not depend on any arbitrary choices, such as the basis in $\bH$;
for details, see
 Nualart \cite{Nualart} or Watanabe \cite{Watanabe}.
In particular, if $\bH=L_2((0,T))$, then $\mfB^{\diamond}(\eta)$ is the
\Ito-Skorokhod integral of $\eta$ in the sense of the Malliavin
calculus.
On the other hand, (\ref{ito-chaos}), if considered as a
formal series, allows the extension of
 $\mfB^{\diamond}$
to weighted chaos spaces, similar to those considered in
\cite{LR2, LR1}; we leave this extension to an interested reader.

\begin{remark}
It is also possible to define
\bel{def:strati}
\mfB^{\circ}(\eta)=\sum_{\alpha\in \cI}
\mfB(\eta_{\alpha})\cdot\xi_{\alpha},
\ee
where $\cdot$ is the usual product.
To understand the structure of this  integral,
 note that the Malliavin derivative $\bD$ of $\xi_{\alpha}$
satisfies
$$
\bD \xi_{\alpha}=\sum_{k=1}^{\infty}
\sqrt{\alpha_k}\xi_{\alpha-\epsilon_k}m_k;
$$
this follows directly from the definition of
$\bD$ \cite[Definition 1.2.1]{Nualart}
and the relation $H_n'(x)=nH_{n-1}(x)$.
On the other hand, direct computations show that
$$
\mfB^{\circ}(\eta)=\sum_{\alpha\in \cI}
\sum_{k=1}^{\infty} \eta_{\alpha,k}\,\xi_k \xi_{\alpha},
$$
 and,  using the following
property of the Hermite polynomials,
$$
H_1(x)H_n(x)=H_{n+1}(x)+nH_{n-1}(x),
$$
we get
\bel{strat-chaos0}
\xi_k\xi_{\alpha}=\left(
\prod_{j\not=k}\frac{H_{\alpha_j}(\xi_j)}{\sqrt{\alpha_j!}}
\right)\frac{H_1(\xi_k)H_{\alpha_k}(\xi_k)}{\sqrt{\alpha_k!}}=
\sqrt{\alpha_k+1}\, \xi_{\alpha+\epsilon_k}+
\sqrt{\alpha_k}\xi_{\alpha-\epsilon_k},
\ee
or
\bel{strat-chaos}
\mfB^{\circ}(\eta)=
\sum_{\alpha\in \cI}
\left( \sum_{k=1}^{\infty}\left(
\sqrt{\alpha_k}\eta_{\alpha-\epsilon_k,k}+
\sqrt{\alpha_k+1}\eta_{\alpha+\epsilon_k,k}\right)\right) \xi_{\alpha}.
\ee
As a result, we use (\ref{strat-chaos0}) to re-write
 (\ref{strat-chaos}) as
\bel{strat-chaos1}
\mfB^{\circ}(\eta)=\mfB^{\diamond}(\eta)+\sum_{\alpha\in \cI}
(\eta_{\alpha},\bD \xi_{\alpha})_{\bH}.
\ee
That is, $\mfB^{\circ}$ is a sum of the \Ito-Skorokhod integral
plus the trace of the Malliavin derivative --- a
representation characteristic of the Stratonovich-type integrals
\cite{Nualart}.
In particular, if $\mfB$ is a white noise over $L_2((0,T))$
and $W(t)=\mfB(\chi_t)$, then
$\mfB^{\circ}(W)=W^2(T)=\int_0^TW(t)\circ dW(t)$.
More generally,
 if $\eta=\eta(t)$
is in the domain of the Malliavin derivative, then
\bel{strat-chaos2}
\mfB^{\circ}(\eta)=\mfB^{\diamond}(\eta)+\int_0^T\bD_t \eta dt,
\ee
where
$$
\bD_t \eta=\sum_{\alpha\in \cI}\eta_{\alpha}(t)
\left(\sum_{k=1}^{\infty}\sqrt{\alpha_k}\,\xi_{\alpha-\epsilon_k}m_k(t)
\right);
$$
unlike the \Ito-Skorokhod integral, though, condition
  (\ref{D12}) is not enough to ensure the existence of
  $\mfB^{\circ}(\eta)$
  as an element of $\HX$. When $\bH$ is the Hilbert space of
  functions on an interval $[0,T]$, square integrable
   with respect to a (not necessarily Lebesque)
    measure $\mu$, the sufficient conditions for
    the Stratonovich integrability
    are discussed in \cite[Chpater 3]{Nualart}.
 \end{remark}

 Let $\mfX$ be a zero-mean non-degenerate
 generalized Gaussian field over a separable Hilbert space $\bH$.
 As we mentioned earlier, by Theorem \ref{th:wn}(b) on page
 \pageref{th:wn},
  $\mfX$ is a white
 noise over a bigger Hilbert space $\bH_{\cR}$, and then
  $\mfX^{\diamond}(\eta)$ can be defined using (\ref{def:itoi}). If the
 space $\bH_{\cR}$ is difficult to describe,
 one can use representation (\ref{wn11}) from
  Theorem \ref{th:wn}(a) and derive an equivalent
 formula for the stochastic integral:
 \bel{def:itoi111}
 \mfX^{\diamond}(\eta)=\mfB^{\diamond}(\cK^* \eta).
 \ee
 for every $(\mfB, \bH)$-admissible $\eta$.

 Unlike (\ref{def:itoi}), representation (\ref{def:itoi111}) is not
 intrinsic: the operator $\cK^*$ and the white noise $\mfB$ are not
 uniquely determined by $\mfX$. On the other hand, in many examples,
 such as fractional Brownian motion with the Hurst parameter bigger than $1/2$,
 it is possible to take $\bH=L_2((0,T))$, and then
 (\ref{def:itoi111}) becomes more convenient than (\ref{def:itoi}).
 To derive the chaos expansion of $\mfX^{\diamond}(\eta)$ using
 (\ref{def:itoi111}), fix  an orthonormal basis $\{m_k,\, k\geq 1\}$
 in $\bH$, define $\xi_k=\mfB(m_k)$,
 and consider the corresponding orthonormal basis
 $\{\xi_{\alpha}, \alpha\in \cI\}$ in $\HB$
  constructed according to (\ref{xi-al}).
  It follows from (\ref{ito-chaos}) that
 \bel{ito-chaos-alt}
\mfX^{\diamond}(\eta)=\sum_{\alpha\in \cI}\left( \sum_{k=1}^{\infty}
\sqrt{\alpha_k}\,\wt{\eta}_{\alpha-\epsilon_k,k}\right) \xi_{\alpha},
\ee
where
$$
 \wt{\eta}_{k,\alpha}=\bE\big((\cK^*\eta,m_k)_{\bH}\,\xi_{\alpha}\big).
 $$

If $\bH=L_2((0,T))$, then (\ref{ito-chaos-alt}) becomes
\bel{st-int-25}
\mfX^{\diamond}(\eta)=\sum_{\alpha\in \cI}
\left(\sum_{k\geq 1} \sqrt{\alpha_k}
\left(\int_0^T \eta_{\alpha-\epsilon_k}(t) (\cK m_k)(t) dt\right)\,
\right)\, \xi_{\alpha},
\ee
where $\eta_{\alpha}(t)=\bE\big(\eta(t) \, \xi_{\alpha}\big)$.
In this case,
by analogy with the Brownian motion,
$\int_0^T \eta(s)\diamond dX(s)$
can be an
alternative notation for $\mfX^{\diamond}(\eta)$,
where $X(t)$ is the associated process of $\mfX$.

\section{Stochastic Evolution Equations with
Closed-Form Solutions}
\label{sec:see1}

In this section we consider stochastic differential equations
driven by a white noise $\mfB$ or some other
 zero-mean generalized Gaussian random field $\mfX$ over a
 Hilbert space $\bH$.
To  introduce time
evolution into the stochastic integral, we
use the function $\chi_t$, the indicator function of the interval
$[0,t]$, and define time-dependent stochastic integrals
\bel{time-evol}
\mfB^{\diamond}_t(\eta):=\mfB^{\diamond}(\eta\chi_t),\ \ \ \
\mfX^{\diamond}_t(\eta):=\mfX^{\diamond}(\eta\chi_t).
\ee
These definitions put an obvious restriction on the Hilbert space
$\bH$, which we call Property I: {\em  $\bH$
contains $\chi_t$, $t\in [0,T]$,
 and, for every $\eta\in \bH$ and every fixed $t$,
the (point-wise) product $\eta\chi_t$
is defined and belongs to $\bH$.} There is a more significant
restriction on $\bH$, which we illustrate on the following
equation:
\bel{sde00}
u(t)=1+ \mfB^{\diamond}_t(u),\ 0\leq t \leq T,
\ee
where $\mfB$ is white noise over a Hilbert space $\bH$ with Property I.
Let us assume that the solution belongs to $\HB$ so that
$u(t)=\sum_{\alpha \in \cI} u_{\alpha}(t)\xi_{\alpha}$ and
each $u_{\alpha}$ is an element of $\bH$. By (\ref{time-evol}),
we can re-write (\ref{sde00}) as
\bel{sde01}
u(t)=1+\mfB^{\diamond}(u\chi_t),
\ee
and then (\ref{ito-chaos}) implies
\bel{sde02}
u_{\alpha}(t)=1+\sum_{k=1}^{\infty}(u_{\alpha-\epsilon_k}\chi_t,m_k)_{\bH}.
\ee
Thus, the expression $(u_{\alpha-\epsilon_k}\chi_t,m_k)_{\bH}$, as a
function of $t$, must be an element of $\bH$, and the
Hilbert space $\bH$ must have another special property,
which we call Property II: {\em for every
$f,g\in \bH$, the inner product $(f\chi_t,g)_{\bH}$, as a function of
$t$, is an element of $\bH$}. The space $\bH_R$ corresponding to
a zero-mean Gaussian process with covariance $R$ contains
step functions by definition, but the point-wise multiplication is a
more delicate issue, especially if $\bH_R$ is smaller than
$L_2((0,T))$. On the other hand,
the space
$L_2((0,T),\mu)$, with $\mu((0,T))<\infty$, has both
Property I and Property II, which follows from the Cauchy-Schwartz
inequality.

\begin{theorem} \label{th:sode1}
If $\mfX$ is a zero-mean generalized Gaussian field over $L_2((0,T))$, then
the solution of the equation
\bel{sde2}
u(t)=1+\mfX^{\diamond}_t(u)
\ee
is unique in $L_2((0,T); \HX)$ and is given by
\bel{stoch-exp1}
u(t)=e^{\diamond X(t)},
\ee
where $e^{\diamond}$ is the Wick exponential function (\ref{Wick-exp})
 and $X(t)=\mfX(\chi_t)$ is the associated process of $\mfX$.
\end{theorem}

\pr Let $\mfX(f)=\mfB(\cK^* f)$ be a white noise
representation of $\mfX$ over $L_2((0,T))$.
We start by establishing uniqueness of solution
 in $L_2((0,T);\HB)$, which, because of
the inclusion $\HX\subseteq\HB$, is even stronger.
  By linearity, the difference
$Y(t)$ of two solutions of (\ref{sde2}) satisfies $Y(t)=\mfX^{\diamond}_t(Y).$
If $Y(t)=\sum\limits_{\alpha\in \cI} y_{\alpha}(t)\xi_{\alpha}$, then
(\ref{st-int-25}) implies
\bel{s-syst1}
y_{\alpha}(t)=\sum_{k\geq 1}\sqrt{\alpha_k}
\int_0^ty_{\alpha-\epsilon_k}(s)\wt{m}_k(s)ds,
\ee
where $\wt{m}_k=\cK m_k$. In particular, if $|\alpha|=0$, then
$y_{\alpha}(t)=0$ for all $t$. By induction on $|\alpha|$,
$y_{\alpha}(t)=0$ for all $\alpha \in \cI$: if $y_{\alpha}=0$
for all $\alpha$ with $|\alpha|=n$, then, since
$|\alpha-\epsilon_k|=|\alpha|-1$, equality (\ref{s-syst1})
implies $y_{\alpha}=0$ for all  $\alpha$ with $|\alpha|=n+1$.

To establish (\ref{stoch-exp1}), let
$$
\wt{M}_k(t)=\int_0^t (\cK m_k)(s)ds.
$$
By (\ref{wn:wp}) and \eqref{wne1},
$$
X(t)=\sum_{k=1}^{\infty} \wt{M}_k(t)\xi_k,
$$
and, because of the independence of $\xi_k$ for different $k$,
$$
e^{\diamond X(t)}=\prod_{k\geq 1} e^{\diamond \wt{M}_k(t)\xi_k}=
\sum_{\alpha\in \cI} \frac{\wt{M}^{\alpha}(t)}{\sqrt{\alpha!}}\, \xi_{\alpha},
$$
where
$$
\wt{M}^{\alpha}(t)=\prod_{k=1}^{\infty} \wt{M}_k^{\alpha_k}(t).
$$
Similar to (\ref{s-syst1}), we conclude that if the solution $u=u(t)$ has the
chaos expansion $u(t)=\sum\limits_{\alpha\in \cI}
u_{\alpha}(t)\xi_{\alpha}$, then $u_{\alpha}(t)=1$ if
$|\alpha|=0$ and
\bel{s-syst12}
u_{\alpha}(t)=\sum_{k\geq 1}\sqrt{\alpha_k}
\int_0^tu_{\alpha-\epsilon_k}(s)\wt{m}_k(s)ds,
\ee
if $|\alpha|>0$. Then direct computations show that
$$
u_{\alpha}(t)=\frac{\wt{M}^{\alpha}(t)}{\sqrt{\alpha!}}, \ |\alpha|\geq 1,
$$
satisfies (\ref{s-syst12}):
\begin{equation*}
\begin{split}
\frac{d u_{\alpha}(t)}{dt}&=\frac{1}{\sqrt{\alpha!}}\,\frac{d}{dt}\prod_{k=1}^{\infty}
\wt{M}_k^{\alpha_k}(t)=\frac{1}{\sqrt{\alpha!}}
\sum_{k=1}^{\infty} \alpha_k \wt{M}_k^{\alpha_k-1}(t)\wt{m}_k(t)\prod_{j\not=k}
\wt{M}_j^{\alpha_j}(t)\\
&=\sum_{k=1}^{\infty}\sqrt{\alpha_k}\, \wt{m}_k(t)
 \frac{\wt{M}^{\alpha-\varepsilon_k}(t)}{\sqrt{(\alpha-\varepsilon_k)!}}=
 \sum_{k=1}^{\infty}\sqrt{\alpha_k}\, m_k(t) u_{\alpha-\varepsilon_k}(t).
 \end{split}
 \end{equation*}
\endproof

\begin{corollary}
\label{c-geom}
The solution of
$$
u(t)=u_0+\int_0^t a(s)u(s)ds+\sum_{k=1}^N \mfX_t^{\diamond}
(\sigma_ku_k)
$$
where $\mfX_k$ are independent generalized Gaussian random fields
over $L_2((0,T))$ and $a\in L_1((0,T)),\ \sigma_k\in L_2((0,T))$
are non-random, is
$$
u(t)=u_0\exp\left(\int_0^ta(s)ds\right)\exp^{\diamond}
\left(\sum_{k=1}^NX_{k}(t)\right),
$$
where $X_k(t)=\mfX_k(\sigma_k\chi_t)$.
\end{corollary}

Theorem \ref{th:sode1} is a generalization of the
familia result that the geometric Brownian motion
$u(t)=e^{W(t)-(t/2)}=e^{\diamond W(t)}$ satisfies
$u(t)=1+\int_0^tu(s)dW(s)$: by  (\ref{stoch-exp1}) and (\ref{WickExp}),
  for a class of zero-mean Gaussian processes $X=X(t)$ with
covariance function $R(t,s)$, and with a suitable interpretation of the
stochastic integral, the solution of the equation
$u(t)=1+\int_0^tu(s)\diamond dX(s)$ is
$$
u(t)=e^{X(t)-X(0)-\frac{1}{2}(R(t,t)-2R(t,0)+R(0,0))}.
$$
Note that the proof works without assuming that $\mfX$ is non-degenerate
or $X$ is non-anticipating. On the other hand,
the special interpretation of the integral is essential: if
$X$ is also is a semi-martingale, then the traditional
\Ito~integral $\int_0^T\eta(t)dX(t)$ can also be defined, and the
solution of the corresponding equation
$U(t)=1+\int_0^tu(s)dX(s)$ is the Dolean exponential
$$
U(t)=e^{X(t)-X(0)-\frac{1}{2}\langle X \rangle_t}.
$$
To see that $u$ and $U$ can be different, consider the stable
Ornstein-Uhlenbeck process \eqref{OU},
for which $u(t)=\exp\big(X(t)-(1-e^{-bt})/(2b)\big)$ and
$U(t)=\exp\big(X(t)-t/2\big)$. The reason for this difference is
clear: while $U$ satisfies the equation $dU(t)=-bX(t)U(t)dt+U(t)dW(t)$,
$u(t)$ satisfies $du(t)=-bX(t)\diamond u(t)dt+u(t)dW(t)$
(it is known that $dW=\diamond dW$; see \cite{HOUZ}).

As an application of Theorem \ref{th:sode1}, let us find the
classical, square-integrable
solution of the stochastic partial differential equation
\begin{equation}
\label{SPDE00}
u(t,x)=u_0(x)+a\int_0^t u_{xx}(s,x)ds+
\sigma \mfX^{\diamond}_t(u_x(\cdot,x)),\ t\geq 0,\ x\in \bR,
\end{equation}
with a smooth compactly supported initial condition $u_0$ and constant
$a>0$, $\sigma\in \bR$.
Let
$$
\widehat{u}(t,y)=\frac{1}{\sqrt{2\pi}}
\int_{\bR} e^{-\mathfrak{i} xy}u(t,x)dx,\ \mathfrak{i}=\sqrt{-1},
$$
be the Fourier transform of $u(t,x)$. Then, by linearity,
$$
\widehat{u}(t,y)=\widehat{u}_0(y)-ay^2\int_0^t\widehat{u}(s,y)ds
+\mathfrak{i} \sigma \mfX_t(\widehat{u}(\cdot,y)),
$$
and therefore
$$
\widehat{u}(t,y)=\widehat{u}_0\exp\left(
-ay^2t+\frac{1}{2}\sigma^2y^2R(t,t)+\mathfrak{i}y\sigma X(t)
\right),
$$
where $X(t)$ is the associated process of $\mfX$ and $R(t,s)=
\mathbb{E}(X(t)X(s))$. With the notation
$$
r(t)=at-\frac{\sigma^2}{2}R(t,t),
$$
the classical solution of \eqref{SPDE00} becomes
$$
u(t,x)=\bar{u}(t,x-\sigma X(t)),
$$
where
$$
\bar{u}(t,x)=\frac{1}{\sqrt{4\pi r(t)}}\int_{\bR}
\exp\left(-\frac{(x-y)^2}{4r(t)}\right) u_0(y)dy.
$$
In particular, we get the following {\em parabolicity} or
{\em non-explosion} condition for \eqref{SPDE00}:
\begin{equation}
\label{nonexpl}
at\geq\frac{\sigma^2}{2}R(t,t).
\end{equation}
If  $X(t)=W(t)$, then  $R(t,t)=t$ and
 \eqref{nonexpl} becomes
 \begin{equation}
\label{parab-cl}
2a\geq \sigma^2.
\end{equation}
If  $X(t)=W^H(t)$, $H>1/2$, then  $R(t,t)=t^{2H}$ and
\eqref{nonexpl} becomes
\begin{equation}
\label{parab-fBM}
t^{2H-1}\leq 2a/\sigma^2;
\end{equation}
 this condition also appears in  \cite{Dunc1}.

For the stable  Ornstein-Uhlenbeck process \eqref{OU},
$R(t,t)=(1-e^{-2bt})/(2b)$,
 and condition \eqref{nonexpl} becomes
$$
a\geq\frac{\sigma^2}{4bt}\left(1-e^{-2bt}\right),
$$
which is  equivalent to \eqref{parab-cl}.

For the unstable Ornstein-Uhlenbeck process \eqref{OU-u},
condition \eqref{nonexpl} is
\begin{equation}
\label{OU-uu}
at\geq \frac{\sigma^2}{4b}\left(e^{2bt}-1\right).
\end{equation}
If \eqref{parab-cl} holds, then \eqref{OU-uu} holds
for sufficiently small $t$; if \eqref{parab-cl} fails, so does
\eqref{OU-uu} for all $t\geq 0$.

 The traditional \Ito~version  of
 \eqref{SPDE00} with the Ornstein-Uhlenbeck process is
$$
du(t,x)=(au_{xx}(t,x)\pm b\sigma X(t)u_x(t))dt+\sigma u_x dW(t),
$$
This equation
 is well-posed if and only \eqref{parab-cl} holds,
 and this condition does not depend on $b$.

We can also write the classical, square-integrable solution of
\begin{equation}
\label{SPDE01}
u(t,x)=u_0(x)+\int_0^t a(s)u_{xx}(s,x)ds+
 \mfX^{\diamond}_t(\sigma u_x(\cdot,x)),\ t\geq 0,\ x\in \bR,
\end{equation}
with a smooth compactly supported initial condition $u_0$ and
continuous functions $a(t), \ \sigma(t)$.
Indeed, let
\begin{equation*}
\begin{split}
&A(t)=\int_0^ta(s)ds,
 \ X_{\sigma}(t)=\mfX(\sigma\chi_t),\\
 & \ R_{\sigma}(t,s)=\mathbb{E}(X_{\sigma}(t)X_{\sigma}(s)),
   \ r_{\sigma}(t)=A(t)-\frac{1}{2}R_{\sigma}(t,t).
\end{split}
\end{equation*}
Then
$$
u(t,x)=\bar{u}(t,x - X_{\sigma}(t)),
$$
where
$$
\bar{u}(t,x)=\frac{1}{\sqrt{4\pi r_{\sigma}(t)}}\int_{\bR}
\exp\left(-\frac{(x-y)^2}{4r_{\sigma}(t)}\right) u_0(y)dy.
$$
The parabolicity condition for \eqref{SPDE01} is
$$
r_{\sigma}(t)\geq 0,
$$
which in general cannot be simplified much further: when
$\sigma$ depends on time,  there is no
easy connection between $R$ and $R_{\sigma}$.

\section{Chaos Solution of Stochastic Evolution Equations}
\label{sec:see2}

Let $\mfX_{\ell},\ \ell\geq 1$, be a collection of independent,
zero-mean, generalized Gaussian random fields over $L_2((0,T))$
on a probability space $(\Omega, \cF,\mathbb{P})$.
By Theorem \ref{th:wn11} on page \pageref{th:wn11},
there is a collection
$\{W_{\ell}, \ell\geq 1\}$ of independent Wiener processes
and $\{\cK_{\ell},\ \ell\geq 1\}$ of bounded linear operators on
$L_2((0,T))$ such that
\begin{equation}
\label{rf-rep}
\mfX_{\ell}(f)=\int_0^T (\cK_{\ell}^* f)(t)dW_{\ell}(t).
\end{equation}
{\tt With no loss of generality, we assume that
the sigma-algebra $\cF$ is \\ generated by the
random variables $W_{\ell}(t)$, $\ell\geq 1$, $t\in [0,T]$.}
Note that $X_{\ell}(t)=\mfX_{\ell}(\chi_t)$ is not necessarily
adapted to filtration of the corresponding Wiener process
$W_{\ell}(t)$.

Introduce the following objects:
\begin{enumerate}
\item $(X,H,X')$,  a  triple of Hilbert spaces
such that $X'$ is the dual of $X$ relative to the inner product
in $H$. {\tt To simplify the notations, we use $(\cdot,\cdot)$ to
denote both the inner product in $H$ and the duality between
$X$ \\ and $X'$.}
\item $A$, a  bounded linear operator
from $L_2((0,T);X)$ to $L_2((0,T);X').$
\item $M_{\ell}$, $\ell\geq 1$, a collection of
bounded linear operators from $L_2((0,T);X)$ to $L_2((0,T);X').$
\item $u_0\in L_2(\Omega;H)$, $f\in L_2(\Omega;L_2((0,T);X'))$,
$g_{\ell}\in L_2(\Omega;L_2((0,T);X'))$.
\item The Fourier cosine basis in $L_2((0,T))$:
\begin{equation}\label{basis}
 m_1(s)\!=\!\frac 1{\sqrt{T}};\  m_k(t)\!=
 \!\sqrt{\frac{2}{T} }
 \cos \left( \frac{\pi (k-1) t}{T} \right), \, k>1; \
0\leq t \leq T.
\end{equation}
Also define  $\widetilde{m}_{k\ell}(t)=(\cK_{\ell} m_k)(t)$.
\item $\cJ$, a collection of multi-indices $\bola=
(\alpha_{k\ell},\ k,\ell\geq 1)$ with non-negative
integer entries $\alpha_{k\ell}$ and the finite length
$|\bola|=\sum_{k,\ell}\alpha_{k\ell}<\infty$. The two special
multi-indices are $\boz$ with all zero entries and
$\bole(ij)$ with $\bole_{ij}(ij)=1$ and $\bole_{k\ell}(ij)=0$
 otherwise. We also write $\bola!=\prod_{k,\ell}\alpha_{k\ell}!$.
\item The random variables $\xi_{k\ell}=\int_0^Tm_k(t)dW_{\ell}(t)$
and
\bel{eq:xi-al}
\xi_{\bola}=
\prod_{k,\ell\geq 1}
\frac{H_{\alpha_{k\ell}}(\xi_{k\ell})}
{\sqrt{\alpha_{k\ell}!}},
 \ee
where, for an integer $n\geq 0$,
   $H_n=H_n(t)$ is the $n$-th {Hermite polynomial}
   \eqref{eq:HerPol}.
   By Theorem \ref{th:basis}, the collection
   $\{\xi_{\bola},\ \bola\in \cJ\}$, is an orthonormal basis
   in $L_2(\Omega)$.
\item the {\tt characteristic set}
of $\bola$ with $|\bola|=n$:
$S_{\bola}=\{(k_1, \ell_1), \ldots, (k_n,\ell_n)\}$,
where  $k_1\leq k_2\leq\cdots\leq k_n$,
$\ell_i\leq \ell_{i+1}$ if $k_i=k_{i+1}$, and
\begin{equation}
\label{xial-alt}
\xi_{\bola}= \frac{\xi_{k_{1}\ell_1}
\diamond\xi_{k_{2}\ell_2}
\diamond\cdots\diamond\xi_{k_{n}\ell_n}}
{\sqrt{\bola!}}.
\end{equation}
\end{enumerate}

Consider the following stochastic evolution equation:
\begin{equation}
\label{evolII.1}
u(t)=u_0+\int_0^tA u(s)ds+\int_0^tf(s)ds+
\sum_{\ell\geq 1} \mfX_{\ell,t}^{\diamond}( M_{\ell}u+g_{\ell}).
\end{equation}

Assume that the random element $u_0$ and the
processes $u,f,g_{\ell}$ have chaos expansions
\begin{equation}
\label{expan000}
\begin{split}
u_0=\sum_{\bola\in \cJ} u_{0,\bola}\xi_{\bola},\
f(t)&=\sum_{\bola\in \cJ} f_{\bola}(t)\xi_{\bola},\
g_{\ell}(t)=\sum_{\bola\in \cJ} g_{\ell,\bola}(t)\xi_{\bola},\\
u(t)&=\sum_{\bola\in \cJ} u_{\bola}(t)\xi_{\bola}.
\end{split}
\end{equation}
Substituting into \eqref{evolII.1} and using \eqref{st-int-25},
we conclude that the (deterministic) functions $u_{\bola}(t)$
 satisfy
 \begin{equation}
 \label{eq:propag}
 \begin{split}
 u_{\bola}(t)&=u_{0,\bola}+
 \int_0^t(Au_{\bola}(s)+f_{\bola}(s))ds\\
 &+
 \sum_{k,\ell} \sqrt{\alpha_{k\ell}}\int_0^t
 (M_{\ell}u_{\bola-\bole(k,\ell)}(s)+
 g_{\ell,\bola-\bole(k,\ell)}(s))\wt{m}_{k\ell}(s)ds.
 \end{split}
 \end{equation}

\begin{definition}
A {\tt chaos solution} of equation \eqref{evolII.1}
is a formal series $u(t)=\sum_{\bola\in \cJ} u_{\bola}(t)\xi_{\bola}$,
where the random variables
$\xi_{\bola}$ are defined by \eqref{eq:xi-al} and the deterministic
functions
$u_{\bola}$ satisfy \eqref{eq:propag}.
\end{definition}

\begin{theorem}
Assume that
\begin{itemize}
\item
  for every $U_0\in H$ and $F\in L_2((0,T);X')$,
the deterministic evolution equation
\begin{equation}
\label{det-eq}
U(t)=U_0+\int_0^t AU(s)ds+\int_0^t F(s)ds
\end{equation}
has a unique solution $U\in L_2((0,T); X)$;
\item
\begin{equation}
\mathrm{ess}\sup_{t\in (0,T)}
|\wt{m}_{k\ell}(t)|<\infty \label{sup-m}
\end{equation}
\end{itemize}
Then equation \eqref{evolII.1} has a unique chaos solution
and every $u_{\bola}$ is an element of
$L_2((0,T);X)$.
\end{theorem}

\begin{proof}
Note that, for $|\bola|=0$, \eqref{eq:propag} is
\begin{equation}
\label{eq:II00}
u_{\boz}(t)=u_{0,\boz}+\int_0^tAu_{\boz}(s)ds+
\int_0^tf_{\boz}(s)ds;
\end{equation}
by assumption, equation \eqref{eq:II00} has
a unique solution  $u_{\boz}\in L_2((0,T); X)$.

We now proceed by induction on $|\bola|$.
Assume that \eqref{eq:propag} has a unique
solution $u_{\bola}\in L_2((0,T); X)$ for all
$\bola$ with $|\bola|\leq n$. Then, for
$\bola$ with $|\bola|=n+1$ we have
$|\bola-\bole(k,\ell)|=n$ so that
$M_{\ell}u_{\bola-\bole(k,\ell)}\wt{m}_{k\ell}\in L_2((0,T); X')$.
With only finitely many terms in the sum
on the right-hand side of \eqref{eq:propag}, the assumptions
of the theorem now imply that \eqref{eq:propag}
has a unique
solution $u_{\bola}\in L_2((0,T); X)$ for all
$\bola$ with $|\bola|= n+1$.
\end{proof}

 {\rm In general, condition \eqref{sup-m} is
necessary:  there is no guarantee that
$f\wt{m}_{k\ell}\in L_2((0,T);X')$ for every
$f\in L_2((0,T);X')$. More information about the
operator $A$ makes it possible to remove this
condition; see Theorem \ref{th:main-par} below.

Assume that $\mfX_{\ell}(\chi_t)=\int_0^tK_{\ell}(t,s)dW_{\ell}(s).$
Then a sufficient condition
for \eqref{sup-m} to hold is
\begin{equation}
\label{bnddII}
\mathrm{ess}\sup_{t\in (0,T)}
\left(K_{\ell}(t,t)+ \int_0^t\left|
\frac{\partial K_{\ell}(t,s)}{\partial t}\right|ds\right)
<\infty;
\end{equation}
this is the case for the fractional Brownian motion with
$H>1/2$ and for the Ornstein-Uhlenbeck process
(stable or unstable).}

A theory of solutions of
\eqref{evolII.1} in weighted chaos spaces
can be developed in complete analogy with
\cite{LR2,LR1}; we leave this development to an interested reader.
Instead, we will investigate when the chaos
solution is in fact square-integrable, that is, when no
weights are necessary. This investigation will lead to an extension of
condition \eqref{nonexpl} on page \pageref{nonexpl}
to equations with variable coefficients.

Denote by $\nk_{\ell}$ the norm of the operator $\cK_{\ell}$
in $L_2((0,T))$.

\begin{theorem}
\label{th:main-par}
 Assume that
\begin{enumerate}
\item The initial condition  $u_0$ and the processes
$f,g_{\ell}$ are deterministic and
$$
\sum_{\ell\geq 1}\int_0^T
\mathfrak{K}_{\ell}^2\|g_{\ell}(t)\|_{X'}^2dt<\infty;
$$
\item There exist positive numbers $\delta_A$ and
 $C_A$ such that, for all $v\in X$ and $t\in [0,T]$,
\begin{equation}
\|A(t)v\|_{X'}\leq C_A\|v\|_X,\ \ (A(t)v,v)+
\delta_A\|v\|_X^2\leq C_A\|v\|_H^2.
\label{parab-det}
\end{equation}
\item There exist a non-negative number $\delta_0< \delta_A$ and
a positive number $C_0$ such that, for all $v\in X$ and $t\in [0,T]$,
\begin{equation}
2(A(t)v,v)+\sum_{\ell\geq 1} \nk_{\ell}^2\|M_{\ell}(t)v\|^2_H+
\delta_0\|v\|_X^2\leq C_0\|v\|_H^2.
\label{parab-gen}
\end{equation}
\end{enumerate}
Then the chaos solution of \eqref{evolII.1} satisfies
\begin{equation}
\label{eq:main-ineq}
\begin{split}
\sup_{0<t<T}\bE\|u(t)\|_H^2+\delta_0\int_0^T
\bE\|u(t)\|_X^2dt& \leq C(C_A,\delta_A,C_0,T)
\Big(\|u_0\|_H^2+\int_0^T\|f(t)\|_{X'}^2dt\\
&+
\sum_{\ell\geq 1}\mathfrak{K}_{\ell}^2
\int_0^T\|g_{\ell}(t)\|_{X'}^2dt\Big).
\end{split}
\end{equation}
\end{theorem}

\begin{proof}
By assumptions of the theorem, \eqref{eq:propag} takes the form
\bel{s-system1}
\begin{split}
u_{\boz}(t)&=u_0+\int_0^t A u_{\boz}(s)ds
+\int_0^t f(s)ds,\ |\bola|=0;\\
u_{\bole(ij)}(t)&=
\int_0^t A u_{\bole(ij)}(s)ds+
\int_0^t \Big(
M_{j}u_{\boz}(s)+g_{j}(s)
\Big)\wt{m}_{ij}(s)ds,\ |\bola|=1;\\
u_{\bola}(t)&=
\int_0^t A u_{\bola}(s)ds
+ \sum_{k,\ell=1}^{\infty}\sqrt{\alpha_{k\ell}}
\int_0^tM_{\ell} u_{\bola-\bole(k\ell)}(s)
\wt{m}_{k\ell}(s)ds, \ |\bola|>1.
\end{split}
\ee
By \eqref{parab-det}, the operator $A$ generates a semi-group
$\Phi_{t,s}$, and the solution of \eqref{det-eq} is
$$
U(t)=\Phi_{t,0}U_0+\int_0^t\Phi_{t,s}F(s)ds.
$$
By induction on $|\bola|$ we conclude that if
$|\bola|=n$, $\{(k_1, \ell_1), \ldots, (k_n,\ell_n)\}$
is the characteristic set of $\bola$, and
$\mathcal{P}_n$ is the set of all permutations
 of $\{1,2,\ldots,n\}$, then
\begin{equation}
\begin{split}
\label{eq:ind_sg}
u_{\bola}(t)=
&\frac{1}{\sqrt{\bola!}}\sum_{\sigma \in \cP_n}
\int_0^t\int_0^{s_{n}}\ldots \int_0^{s_2}
\Phi_{t,s_n}M_{\ell_{\sigma(n)}}\cdots
\Phi_{s_2,s_1}\Big(M_{\ell_{\sigma(1)}}u_{\boz}(s_1)\\
&+g_{\ell_{\sigma(1)}}(s_1)\Big)
(\cK_{\ell_{\sigma(k)}} m_{k_{\sigma(n)}})(s_n) \cdots
(\cK_{\ell_{\sigma(1)}} m_{k_{\sigma(1)}})(s_1)
ds^n,
\end{split}
\end{equation}
where $ds^n=ds_1\ldots ds_n$.
We then re-write (\ref{eq:ind_sg}) as
\bel{eq:ind_sg1}
u_{\bola}(t)=
\int_{[0,T]^n}G(t,\ell^{(n)};s^{(n)})
\overline{m}_{\bola}(s^{(n)})ds^n,
\ee
where
\begin{equation}
\begin{split}
G(t,\ell^{(n)};s^{(n)})
&=\frac{1}{\sqrt{n!}}\sum_{\sigma\in \mathcal{P}_n}
\Phi_{t-s_{\sigma(n)}}M_{\ell_n}\cdots
\Phi_{s_{\sigma(2)}-s_{\sigma(1)}}
\Big( M_{\ell_1}u_{\boz}(s_{\sigma(1)})\\
&+g_{\ell_1}(s_{\sigma(1)})\Big)
\chi_{s_{\sigma(2)}}(s_{\sigma(1)})
\cdots\chi_t(s_{\sigma(n)}),
\end{split}
\end{equation}
and
\begin{equation}
\overline{m}_{\bola}(s^{(n)})
=\frac{1}{\sqrt{\bola!\,n!}}\sum_{\sigma \in\mathcal{P}_n}
 (\cK_{\ell_1} h_{k_1})(s_{\sigma(1)})
\cdots (\cK_{\ell_n} h_{k_n})(s_{\sigma(n)}).
\end{equation}
From (\ref{eq:ind_sg1}) and the definition of the
function $G$, we conclude that
\begin{equation}
\label{mult-int}
\begin{split}
\sum_{\stackrel{\scriptstyle{\bola\in\cJ}}{|\bola|=n}}
&\|u_{\bola}(t)\|^2_H
\leq \sum_{\ell_1,\ldots,\ell_n=1}^{\infty}
\left(\prod_{j=1}^n \mathfrak{K}_{\ell_j}^2\right)
 \int_0^t\int_0^{s_{n}}\ldots \int_0^{s_2}
 \\ & \left\|\Phi_{t,s_n}M_{\ell_{n}}\cdots
\Phi_{s_2,s_1}\Big(M_{\ell_{1}}u_{\boz}(s_1)\right.
\left.+g_{\ell_1}(s_1)\Big)\right\|_H^2
ds^n.
\end{split}
\end{equation}
Similarly,
\begin{equation}
\label{mult-intTX}
\begin{split}
\sum_{\stackrel{\scriptstyle{\bola\in\cJ}}{|\bola|=n}}
&\int_0^t\|u_{\bola}(t)\|^2_Xds
\leq \sum_{\ell_1,\ldots,\ell_n=1}^{\infty}
\left(\prod_{j=1}^n \mathfrak{K}_{\ell_j}^2\right)
 \int_0^t\int_0^s\int_0^{s_{n}}\ldots \int_0^{s_2}
 \\ & \left\|\Phi_{s,s_n}M_{\ell_{n}}\cdots
\Phi_{s_2,s_1}\Big(M_{\ell_{1}}u_{\boz}(s_1)\right.
\left.+g_{\ell_1}(s_1)\Big)\right\|_X^2
ds^nds.
\end{split}
\end{equation}
For $n\geq 1$, denote by  $F_{n}^H(t)$ and $F_{n}^X(t)$
the right-hand sides of \eqref{mult-int} and \eqref{mult-intTX},
respectively. For $n=0$, define
$F_0^H(t)= \|u_{\boz}(t)\|_H^2$,
$F_0^X(t)=\int_0^t\|u_{\boz}(s)\|_X^2ds$. Then
\bel{IIaux00}
\bE\|u(t)\|_H^2\leq \sum_{n=0}^{\infty}F_n^H(t),\
\int_0^T\bE\|u(t)\|_X^2dt \leq \sum_{n=0}^{\infty} F_{n}^X(T).
\ee
For brevity, introduce the notation
$$
\bI= \|u_0\|_H^2+\int_0^T\|f(t)\|_{X'}^2dt
+\sum_{\ell\geq 1}\mathfrak{K}_{\ell}^2
\int_0^T\|g_{\ell}(t)\|_{X'}^2dt.
$$
Then assumption \eqref{parab-det} implies
 \bel{IIaux1}
 \sup_{0<t<T}F_0^H(t) + \delta_0 F_0^X(T)
 \leq C_1(C_A,\delta_A,T)\,\bI.
\ee

For $n\geq 1$, we find using \eqref{parab-gen}
that
\begin{equation}
\label{eq:LinHp3}%
\begin{split}
&  \frac{dF_{n}^H(t)}{dt}+\delta_0\,
\frac{dF_{n}^X(t)}{dt}
 \leq C_{0}F_{n}^H(t)\\
+  &  \sum_{\ell_{1}, \ldots, \ell_{n}\geq1}
\left(\prod_{j=1}^n \mathfrak{K}_{\ell_j}^2\right)
 \int_{0}^{t} \int_{0}^{s_{n-1}}%
\ldots\int_{0}^{s_{2}} \\
&\|{{{M}}}_{\ell_{n}} \Phi_{t,s_{n-1}%
}{{{M}}}_{\ell_{n-1}} \ldots
\Phi_{s_2,s_1}\Big(M_{\ell_{1}}u_{\boz}(s_1)
+g_{\ell_1}(s_1)\Big)\|_{H}^{2}ds^{n-1}\\
-  &  \sum_{\ell_{1}, \ldots, \ell_{n+1}\geq1}
\left(\prod_{j=1}^{n+1} \mathfrak{K}_{\ell_j}^2\right)
\int_{0}^{t} \int_{0}^{s_{n}}%
\ldots\int_{0}^{s_{2}} \\
&\|{{{M}}}_{\ell_{n+1}} \Phi_{t,s_{n}%
}{{{M}}}_{\ell_{n}}
 \ldots\Phi_{s_2,s_1}\Big(M_{\ell_{1}}u_{\boz}(s_1)
+g_{\ell_1}(s_1)\Big)\|_{H}^{2}ds^{n}.
\end{split}
\end{equation}
Then, after summation in $n$ and integration in time,
$$
\sum_{n=1}^N\big(F_{n}^H(t)+\delta_0
F_{n}^X(t)\big)\leq C_0\int_0^t
\sum_{n=1}^N F_{n}^H(s)ds + C_2(C_A,\delta_A,T)\,\bI
$$
for every $N\geq 1$.
Applying Gronwal's inequality,
we get  \eqref{eq:main-ineq} from \eqref{IIaux00}.
Theorem \ref{th:main-par} is proved.
\end{proof}

\begin{example} {\rm Let $X$ be the Sobolev space $H^1(\bR^d)$ and
$H=L_2(\bR^d)$. Consider the following equation driven by a
single fractional Brownian motion $W^H$ with $H\geq 1/2$
[there should be no difficulty distinguishing between $H$ a space
and $H$ a Hurst parameter]:
$$
du(t,x)=\sum_{i,j=1}^d a_{ij}(t,x)\frac{\partial^2u(t,x)}{\partial x_i
\partial x_j}\,dt+
\sum_{i=1}^d\sigma_i(t,x)\frac{\partial u(t,x)}{\partial x_i}
\diamond dW^H(t),\ 0<t\leq T.
$$
In this case, $\nk^2=2H\, 2^{1-2H}\,T^{2H-1}$: see
\eqref{fbm-bound} on page \pageref{fbm-bound}.
Condition \eqref{parab-gen} becomes
\bel{parab-fBM1}
\delta_0|y|^2\leq \sum_{i,j=1}^d \Big(
a_{ij}(t,x)-H2^{1-2H}T^{2H-1}\sigma_i(t,x)
\sigma_j(t,x)\Big)y_iy_j\leq C_0|y|^2
\ee
for all $t\in [0,T]$ and all $x,y\in \bR^d$.
Let us now compare \eqref{parab-fBM1} with \eqref{parab-fBM}
on page \pageref{parab-fBM}
if $a,\sigma$ are constants and $d=1$.
If $H=1/2$,  then
 \eqref{parab-fBM1} becomes \eqref{parab-cl}, which is
\eqref{parab-fBM} with $H=1/2$
 (recall that $\delta_0$ can be zero).
If $H>1/2$, then \eqref{parab-fBM1} becomes
$2a/\sigma^2\geq H2^{2-2H}T^{2H-1}$, which is slightly stronger than
\eqref{parab-fBM} because $1<H2^{2-2H}<1.07$ for
$1/2<H<1$.}
\end{example}

\begin{example} {\rm
Let $X(t)=\mfX(\chi_t)$ and consider the equation
\begin{equation}
 \label{eq:ex6.5.1}
du(t,x)=au_{xx}dt+\sigma u_x\diamond dX(t),\ x\in \bR,\
0<t\leq T,
\end{equation}
with constant $a,\sigma$.  Condition \eqref{parab-gen}
 in this case is
 \begin{equation}
 \label{eq:ex6.5.2}
 a\geq \frac{\sigma^2}{2}\, \mathfrak{K}^2.
 \end{equation}
Recall condition \eqref{nonexpl} on page \pageref{nonexpl},
which was derived from the closed-form solution of
equation \eqref{eq:ex6.5.1} and is both necessary and sufficient
for \eqref{eq:ex6.5.1} to
 have a square-integrable solution.
 We conclude that, for equations with constant
coefficients,
\eqref{eq:ex6.5.2} should imply \eqref{nonexpl}, but
not necessarily the other way around.
As a result, comparison of \eqref{nonexpl} and
\eqref{eq:ex6.5.2} produces a lower bound on the
operator norm
$\mathfrak{K}$ for the field $\mfX$
 in terms of the covariance function $R$ of the associated
process $X$:
$$
\mathfrak{K}\geq \sup_{0<t<T}\frac{R(t,t)}{t}.
$$
}
\end{example}

%\bibliographystyle{plain}
%\bibliography{Thesis5}

%\bibliography{Thesis5,Flt} %This is to use multiple data bases.

\end{document}